\documentclass[12pt]{article}
\usepackage{hyperref,color}
\usepackage[margin=.5in]{geometry}
\usepackage{titlesec}
\usepackage{longtable}
\usepackage{amsthm}
\usepackage{amsmath}
\usepackage{amsfonts}
\usepackage{amssymb}
\usepackage{mathrsfs}
\usepackage{authblk}
\usepackage{cite}
\newtheorem{theorem}{Theorem}
\newtheorem{definition}[theorem]{Definition}

\newtheorem{proposition}[theorem]{Proposition}

\usepackage{lipsum}
\usepackage[T1]{fontenc}
\usepackage{wasysym}

\titlespacing\subsection{0in}{\parskip}{\parskip}
\titlespacing\section{0in}{\parskip}{\parskip}

\usepackage{enumitem}
\usepackage{lmodern}

\begin{document}
\title{Sasakian structure associated with a second order ODE and Hamiltonian dynamical systems}

\author[1]{T. Bayrakdar\\ \small{\emph{Faculty of Sciences, Department of Mathematics, Trakya University, Edirne, Turkey}\\ \emph{Email: tunabayraktar@gmail.com} }}

\maketitle

\begin{abstract}
We define a contact metric structure on the manifold corresponding to a second order ordinary differential equation $d^2y/dx^2=f(x,y,y')$ and show that the contact metric structure is Sasakian if and only if the 1-form $\frac{1}{2}(dp-fdx)$ defines a Poisson structure. We consider a Hamiltonian dynamical system defined by this Poisson structure and  show that the Hamiltonian vector field, which is a multiple of the Reeb vector field, admits  a compatible bi-Hamiltonian structure for which $f$ can be regarded as a Hamiltonian function. As a particular case, we give a compatible bi-Hamiltonian structure of the Reeb vector field such that the structure equations correspond to the Maurer-Cartan equations of an invariant coframe on the Heisenberg group and the independent variable plays the role of a Hamiltonian function. We also show that the first Chern class of the normal bundle of an integral curve of a multiple of the Reeb vector field vanishes iff $f_x+ff_p = \Psi (x)$ for some $\Psi$.
\end{abstract}

\textbf{Keywords} Sasakian structure, Poisson structure, Hamiltonian systems, ordinary differential equations

\textbf{Mathematics Subject Classification (2010)} 53C25, 53D10, 53D17, 34A26,
58A15 

\section{Introduction}
An autonomous dynamical system $\dot{x}=v\left(x(t)\right)$ on a smooth manifold $\Sigma$ endowed with a Poisson structure is said to be Hamiltonian if it can be written in the form
\begin{equation}
\label{e1}
v=\Omega (dH,\cdot),
\end{equation} 
where $H$ is the Hamiltonian function and $\Omega$ is the Poisson bi-vector. In three dimensions, the differential of Hamiltonian function and Poisson 1-form corresponding to the Poisson structure define codimension one foliations contrary to the Hamiltonian vector field which is not necessarily holonomic and it may define a non-integrable contact distribution.  In this respect it will be convenient to take into account an additional geometric structure, such as a contact metric structure or in particular a Sasakian structure for the investigation of a Hamiltonian dynamical system identified with a non-holonomic vector field.

The notion of  a normal contact metric structure or a Sasakian structure on an odd dimensional smooth manifold was introduced by Sasaki and Hatakeyama in \cite{Sasaki1962} after the papers \cite{Sasaki1960,Sasaki1961} by exhibiting  certain tensor analogues to the Nijenhuis tensor on an almost complex manifold such that  the contact metric structure is said to be normal provided that this tensor vanishes. Several  works in literature have been devoted to the study of contact metric manifolds  with regard to  their geometric properties, see for example \cite{Hatakeyama1963,Olszak1979,Chern1984,Blair1977,Blair1990,Blair1995,Perrone1990,Perrone1995,Koufogiorgos1993,Geiges1997,Sharma1995}, as well as to their applications in physics. Contact  manifolds (with or without a metric) appear in physics as a contact phase space and provide a geometric description for thermodynamics \cite{Mrugala1991,Mrugala2000,Hernandez1998,Bravetti2015,Hudon2015,Eberard2007}  and for contact Hamiltonian dynamics as an extension of symplectic dynamics \cite{Bravetti2017,Cruz2018}. In particular, three dimensions deserve a special interest on account of the study of  contact metric structures and a geometric formulation of Hamiltonian dynamical systems. In three dimensions a Sasakian structure is recast as a $K$-contact metric structure such that the Reeb vector field of the contact 1-form is Killing \cite{Blair2002}. Concerning a Hamiltonian dynamical system,  a Hamiltonian dynamical system has an elegant representation provided with the musical isomorphisms between covariant and contravariant objects in three dimensions \cite{Gumral1993,Abadoglu2009}.

In this paper we define a contact metric structure on the three-dimensional smooth manifold corresponding to a second order ordinary differential equation (ODE) $d^2y/dx^2=f(x,y,y')$ such that the Reeb vector field is Killing if and only if $f$ does not depend on the variable $y$. Then we see that this condition is equivalent to the existence of the Poisson structure determined by the 1-form $\eta^3=\frac{1}{2}(dp-fdx)$. For this Poisson structure we consider a Hamiltonian dynamical system and show that an integral curve of the Hamiltonian vector field, which is a multiple of the Reeb vector field, corresponds to a geodesic curve.  We also show that this Hamiltonian vector field admits a compatible bi-Hamiltonian structure such that $f$ can be taken as a Hamiltonian function. As a particular case we consider the bi-Hamiltonian structure of  the Reeb vector field and we see that for a second order  ODE of the form $d^2 y/ dx^2 = f(x)$ it is possible to define a compatible bi-Hamiltonian structure in a way that the independent variable plays the role of a Hamiltonian function. This corresponds to the distinguished case when the structure equations for the coframe encoding the second order ODE are the Maurer-Cartan equations for the invariant coframe on the Heisenberg group and the motion is governed by the left-invariant vector field. As a final remark we evaluate the curvature of the connection on the vector bundle whose fibers are annihilated by the contact form, and show that the curvature of this connection and hence the first Chern class of the normal bundle of an integral curve of a multiple of the Reeb vector field vanishes iff $f_x+ff_p = \Psi (x)$ for some function $\Psi$.

\section{Preliminaries}

A \textit{contact manifold} $(M,\eta)$ is a smooth odd-dimensional manifold $M$ endowed with a 1-form $\eta$, so called contact 1-form, satisfying $\eta\wedge (d\eta)^n\neq 0$  globally on  $M$. On a contact manifold $(M,\eta)$, there is a unique vector field $\xi$, called the \textit{Reeb vector field} or \textit{characteristic vector field}   of $\eta$, satisfying $\eta(\xi)=1$ and $d\eta(X,\xi)=0$ for any $X\in \Gamma(TM)$. Here $\Gamma(TM)$ denotes the set of all smooth sections of the tangent bundle $TM$ which is a module over the ring of smooth functions on $M$.  The choice of a contact form is an additional structure on a contact manifold. Scaling
the contact form gives the same contact structure (but a different Reeb vector field). A Riemannian metric $g$ on $(M,\eta)$ is said to be an \textit{associated metric} if there exists an almost contact metric structure  such that  $g\left(X,\phi(Y)\right)=d\eta (X,Y)$ for any $X,Y\in \Gamma(TM)$ \cite{Blair2002}. In other words, a Riemannian metric $g$ on the contact manifold $(M,\eta)$ with the Reeb vector field $\xi$ is said to be an associated metric if 
there exists a tensor field $\phi$ of type (1,1) such that
\begin{equation}\label{CMS}
\phi^2= -\textrm{id}+\eta\otimes\xi, 
\end{equation}
and
\begin{equation}\label{CMS2}
g(\phi(X),\phi(Y))=g(X,Y)-\eta(X)\eta(Y), 
\end{equation}
and
\begin{equation}\label{AdaptedMetric}
d\eta(X,Y)=g\left(X,\phi(Y)\right)
\end{equation}
for any $X,Y\in \Gamma(TM)$.  $(\phi,\xi,\eta, g)$  is called a \textit{contact metric structure}. Together with a contact metric structure,  $(M,\eta)$ is called a \textit{contact metric manifold}  and is denoted by  $(M,\phi,\xi,\eta,g)$. Without the condition (\ref{AdaptedMetric}), $(\phi,\xi,\eta, g)$ defines an almost contact metric structure on $(M,\eta)$.  Note that on an (almost) contact metric manifold we have $\phi(\xi)=0$ and  $\eta\circ \phi =0$. Also setting $\xi=Y$ in (\ref{CMS2}) gives $\eta(X)=g(X,\xi)$.    
On a contact metric manifold the tensor field $h$ of type (1,1) defined by $h=\frac{1}{2}\mathcal{L}_\xi \phi$, where $\mathcal{L}_\xi$ is the Lie derivative with respect to $\xi$, is a symmetric operator, i.e. $g(h(X),Y)=g(X,h(Y))$ for all $X,Y\in \Gamma(TM)$ and satisfies the following properties:
\begin{equation}
h(\xi)=0,~~h\circ\phi=-\phi\circ h,~~\textrm{tr}h=0.
\end{equation}
If $\nabla$ is the Riemannian connection on a contact metric manifold $(M,\phi,\xi,\eta, g)$ compatible with $g$, then we have
\begin{equation}
\nabla_X \xi = -\phi(X)-\phi\left(h(X)\right). 
\end{equation}
A contact metric structure $(\phi,\xi,\eta, g)$ is called \textit{$K$-contact} if $\xi$ is a Killing vector field, that is $\mathcal{L}_\xi g=0$. Note that $\xi$ is a Killing vector field iff the tensor $2h=\mathcal{L}_\xi \phi$  vanishes, see e.g. \cite{Sasaki1962,Blair2002}.
A contact metric structure $(\phi,\xi,\eta, g)$ is called \textit{Sasakian} iff 
\begin{equation}
(\nabla_X \phi)Y = -g(X,Y)\xi-\eta(Y)X. 
\end{equation}
A Sasakian manifold is $K$-contact and the converse is also true in three dimensions. For a detailed account of the contact metric structures we refer to \cite{Blair1976,Blair2002}.

\section{Sasakian structure associated with a second order ODE}

A second order ODE of the form
\begin{equation}\label{second ode}
\frac{d^2y}{dx^2}= f(x,y,y')
\end{equation}
can be regarded as a closed submanifold $\Sigma$ of the second order jet bundle $J^2\pi$ defined to be the zero set of the function $F(x,y,p,q)=q-f(x,y,p)$, that is, $\Sigma= F^{-1}(\{0\})$, where $(x,y,p,q)$ denotes the standard coordinates on $J^2\pi$. The natural inclusion $i:\Sigma\hookrightarrow J^2\pi$ defines an embedding  which is locally given by 
\begin{equation}\label{inc1}
i:(x,y,p) \mapsto\left(x,y,p, q=f(x,y,p)\right), 
\end{equation}
where $(x,y,p)$ is considered as a local coordinate system on $\Sigma$.
The second order jet bundle $J^2\pi$ of maps $\mathbb{R}\rightarrow \mathbb{R}$ is a smooth (fibered) manifold  of all 2-jets of smooth sections of the trivial bundle $\pi : \mathbb{R}\times \mathbb{R}\rightarrow \mathbb{R}$. A section of $\pi$ is a function $\gamma:\mathbb{R}\rightarrow\mathbb{R}\times \mathbb{R}$ such that $ \pi\circ\gamma$ is the identity map on the first factor of $\mathbb{R}\times \mathbb{R}$. If $x$ and $y$ stand for coordinates for the respective factors of $\mathbb{R}\times \mathbb{R}$, then $\gamma$ is given in coordinates by $\gamma(x)=(x,y(x))$ and is identified with the function $y(x)$. The 2-jet of $\gamma$ at a point $x=x_0$ is an equivalence class of sections of $\pi$ having the same Taylor coefficients with $y(x)$ up to order 2 at $x=x_0$ and is denoted by $j_{x_0}^2 \gamma$. The standard coordinate functions $x,y,p,q$ on $J^2\pi$ are defined to be $x(j_{x_0}^2 \gamma)=x_0$, $y(j_{x_0}^2 \gamma)=y(x_0)$, $p(j_{x_0}^2 \gamma)=y'(x_0)$, $q(j_{x_0}^2 \gamma)=y''(x_0)$ and hence the point $j_{x_0}^2 \gamma$ is given by $j_{x_0}^2 \gamma=(x_0,y(x_0),y'(x_0),y''(x_0))$.
The image of a smooth curve $\gamma:I\subset\mathbb{R}\rightarrow \mathbb{R}\times\mathbb{R}; x\mapsto\left(x,y(x)\right)$, which is a local section of $\pi$,  defines a solution  of (\ref{second ode}) if the image $S$ of the map $j^2 \gamma: I\rightarrow J^2\pi; {x}\mapsto j_{x}^2\gamma$ 
is a curve on $\Sigma$ such that  the contact forms
\begin{equation}\label{exterior diff system}
\omega^2=dy-pdx,~~\omega^3=dp-fdx
\end{equation}
on $J^2\pi$ vanish on $S$ when they are pulled-back. Here  the map $j^2 \gamma$ is called the 2-graph of the local section $\gamma$. Accordingly, the solutions of the exterior differential system generated by the 1-forms $\omega^2,\omega^3$ with independence condition $\omega^1=dx\neq 0$
are in one-to-one correspondence with the solutions of (\ref{second ode}). For the details of geometric formulation of differential equations and the exterior differential systems see e.g.  \cite{Vassiliou2000,Saunders1989,Bryant1995}.

Let $U\subset\Sigma$ be a coordinate neighborhood with coordinates $(x,y,p)$ and consider the Riemannian metric $g$ on $U$ defined by
\begin{equation}\label{metric}
g= \sum_i \eta^i\otimes \eta^i,
\end{equation}
where $\eta^i=\frac{1}{2}\omega^i$ for $i=1,2,3$. Note that the local coframe $\eta_g=(\eta^1,\eta^2,\eta^3)^t$ is the dual to the orthonormal frame $(\xi_1,\xi_2,\xi_3)$ of the vector fields 
\begin{equation}\label{vectorfields}
\xi_1=  2\left(\frac{\partial}{\partial x}+p\frac{\partial}{\partial y}+f \frac{\partial}{\partial p}\right),~~\xi_2=2\frac{\partial}{\partial y},~~\xi_3=2\frac{\partial}{\partial p}.
\end{equation}
Here $^t$ stands for the transposition. The structure equations for the coframe $\eta_g$ are then given by
\begin{eqnarray}\label{strcture}
d\eta^1&=&0 \nonumber \\
d\eta^2 &=& 2\eta^1\wedge\eta^3 \\
d\eta^3 &=& 2f_y \eta^1\wedge\eta^2+2f_p \eta^1\wedge\eta^3. \nonumber
\end{eqnarray}
Let $\nabla$ be the metric connection in the tangent bundle $T\Sigma$ compatible with the Riemannian metric (\ref{metric}). We shall find the matrix $\theta$ of 1-forms $\alpha,\beta,\delta$ such that $d\eta_g= -\theta\wedge \eta_g$, i.e.
\begin{equation}
d\left[\begin{array}{c}
\eta^1\\
\eta^2\\
\eta^3
\end{array}\right]=-\left[\begin{array}{rrr}
0 & -\alpha & -\beta \\
\alpha & 0 & -\delta\\
\beta & \delta & 0
\end{array}\right]\wedge \left[\begin{array}{c}
\eta^1\\
\eta^2\\
\eta^3
\end{array}\right].
\end{equation}
The matrix  of 1-forms $\theta$ is called the  \textit{connection form} of $\nabla$ on $U$ and is regarded as a 1-form on $U$ which takes its values in $\mathfrak{so}(3,\mathbb{R})$, the Lie algebra of the orthogonal group. By the same idea exposed in \cite{Bayrakdar2018,OkBayrakdar2019} we have the following: 
\begin{proposition}
	Let $(\Sigma,g)$ be the Riemannian manifold	corresponding to a second order ODE $y''= f(x,y,y')$ and let $\nabla$ be the metric connection compatible with $g$. The  connection form of $\nabla$ on $U$ is determined by
	\begin{equation}
	\theta= \left[\begin{array}{rrr}
	0 & -\alpha &-\beta \\
	\alpha & 0 & -\delta \\
	\beta & \delta & 0 
	\end{array}\right],
	\end{equation}
	where $\alpha, \beta$ and $\delta$ are defined respectively by $\alpha  =  (f_y +1)\eta^3$, $\beta=(f_y +1)\eta^2+2f_p\eta^3$ and $\delta =  -(f_y -1)\eta^1.$
\end{proposition}
By means of the connection form $\theta$ the covariant derivative of $\xi_j$ relative to $X\in \Gamma(T\Sigma)$ is given by
\begin{equation}
\label{conn1}
\nabla_X \xi_j =\sum_i \theta^i_j (X)\xi_i,
\end{equation}
where $\theta^1_2=-\alpha, \theta^1_3=-\beta,\theta^2_3=-\delta$. Accordingly we obtain
\begin{equation}
\nabla_{\xi_1} \xi_1=0, ~\nabla_{\xi_1} \xi_2=-(f_y -1)\xi_3 ,~\nabla_{\xi_1} \xi_3=(f_y -1)\xi_2,
\end{equation}
\begin{equation}\label{geo1}
\nabla_{\xi_2} \xi_1=(f_y +1)\xi_3, \nabla_{\xi_2} \xi_2=0, ~\nabla_{\xi_2} \xi_3=-(f_y +1)\xi_1 ,
\end{equation}
\begin{equation}
\nabla_{\xi_3} \xi_1=(f_y +1)\xi_2+2f_p\xi_3, ~\nabla_{\xi_3} \xi_2=-(f_y +1)\xi_1 ,~\nabla_{\xi_3} \xi_3=-2f_p\xi_1.
\end{equation}
For the details on a connection in a vector bundle and a Lie algebra-valued 1-form, one may consult \cite{Morita2001}.

\subsection{Contact Metric Structure Determined by $\eta^2$}
Since $\eta^2$ is a contact form on $\Sigma$,  $(\Sigma,\eta)$ is a contact manifold with the characteristic vector field $\xi_2=2\frac{\partial}{\partial y}$. It is clear that $\eta^2(\xi_2)=1$ and $d\eta^2(X,\xi_2)=0$ for any $X \in \Gamma(T\Sigma)$. Writing $X=X^i \xi_i$ we see that $X^2=\eta^2(X)=g(X,\xi_2)$ holds for the Riemannian metric  (\ref{metric}). If we introduce the (1,1) tensor field 
\begin{equation}
\phi= \eta^3\otimes \xi_1-\eta^1\otimes \xi_3,
\end{equation}
then we get
\begin{equation}
\phi^2=-\textrm{id}+\eta^2\otimes\xi_2=- (\eta^1\otimes \xi_1+\eta^3\otimes \xi_3).
\end{equation}
It is easy to see that
\begin{equation}
d\eta^2(X,Y)=g\left(X,\phi(Y)\right)=X^1Y^3-Y^1X^3 
\end{equation}
and
\begin{equation}
g(\phi(X),\phi(Y))=g(X,Y)-\eta(X)\eta(Y)=X^1Y^1+X^3Y^3
\end{equation}
are satisfied for any $X,Y\in \Gamma(T\Sigma)$. Accordingly, we have obtained a contact metric structure $(\phi,\xi_2,\eta^2, g)$ on $\Sigma$. Clearly we have $\phi(\xi_2)=0$,  $\eta^2\circ \phi =0$.

Since the Lie derivative $\mathcal{L}_\xi$ is a derivation on the algebra of tensor fields on $\Sigma$, if we use the identities 
\begin{equation}
\mathcal{L}_\xi \omega=\iota_\xi d\omega+d(\iota_\xi \omega)
\end{equation}
and
\begin{equation}\label{voT}
\mathcal{L}_\xi X=[\xi,X]= \nabla_\xi X- \nabla_X \xi,
\end{equation}
for a 1-form $\omega$ and a vector field $X$ on $\Sigma$, we get
\begin{equation}
\mathcal{L}_{\xi_2} \phi = -2f_y \left(\eta^1\otimes \xi_1-\eta^3\otimes \xi_3\right).
\end{equation}
Here $\iota$ denotes the contraction and the identity (\ref{voT}) manifests that the connection $\nabla$ is torsion free. Since $2h=\mathcal{L}_{\xi_2} \phi$ we have
\begin{equation}
h= -f_y \left(\eta^1\otimes \xi_1-\eta^3\otimes \xi_3\right).
\end{equation}
Also the Lie derivative of the metric tensor (\ref{metric}) is obtained as
\begin{equation}
\mathcal{L}_{\xi_2} g = -2f_y \left(\eta^1\otimes \eta^3+\eta^3\otimes \eta^1\right).
\end{equation}
It follows that $\xi_2$ is Killing vector field if and only if $f_y=0$. This is equivalent to say that the contact metric structure $(\phi,\xi_2,\eta^2, g)$ is $K$-contact iff $f_y=0$. 
\begin{theorem}\label{thm2}
	Let $\Sigma$ be the submanifold in $J^2\pi$ corresponding to a  second order ODE of the form $y''= f(x,y')$. Then $(\phi,\xi_2,\eta^2, g)$ defines a Sasakian structure on $\Sigma$, where $\phi,\xi_2, \eta^2,$ and $g$ are given in local coordinates $(x,y,p)$ respectively by
	\begin{equation}
	\phi= (dp-fdx)\otimes (\partial_x+p\partial_y+f\partial_p)-dx\otimes \partial_p,~~\xi_2=2\partial_y,~~\eta^2=\frac{1}{2}(dy-pdx)
	\end{equation}
	and 
	\begin{eqnarray*}
		g &=& \frac{1}{4}\bigg[(1+p^2+f^2)dx\otimes dx-p(dx\otimes dy+dy\otimes dx)-f(dx\otimes dp+dp\otimes dx)\\ 
		&+&dy\otimes dy+dp\otimes dp\bigg].
	\end{eqnarray*}
	
\end{theorem}

\section{Hamiltonian structure}

An autonomous dynamical system $\dot{x}=v(x)$ is a vector field on a smooth manifold $\Sigma$ endowed with a Poisson structure is said to be Hamiltonian if it can be written as
\begin{equation}
\label{e1}
v=\Omega (dH,\cdot).
\end{equation} 
Here $H$ is called Hamiltonian function and $\Omega$ is the Poisson bi-vector, which is a  skew-symmetric, contravariant rank two tensor subjected to the Jacobi identity
\begin{equation}
\label{e2}
[\Omega,\Omega]_{\textrm{SN}}=0,
\end{equation} 
where $[\cdot,\cdot]_{\textrm{SN}}$ denotes the Schouten-Nijenhuis bracket. In a local coordinate system $(x^i)$, $\Omega$ is given by
\begin{equation}
\label{e3}
\Omega= \Omega^{ij} \partial_i\wedge \partial_j,~~~~\partial_i=\frac{\partial}{\partial x^i} 
\end{equation} 
for which the Jacobi identity is given by
\begin{equation}
\label{e4}
\Omega^{i[j}\partial_i\Omega^{kl]}=0,
\end{equation} 
where $[jkl]$ denotes the anti-symmetrization.

Instead of Poisson bi-vector $\Omega$, one can deal with the 1-form  $\mathcal{J}$, so called Poisson 1-form associated with $\Omega$, which is defined by the contraction of the volume form by $\Omega$ \cite{Gumral1993}:
\begin{equation}
\label{e7}
\mathcal{J} =\imath_\Omega \textrm{vol},
\end{equation}
where $\imath_\Omega \textrm{vol}$ is defined by the pairing $\langle \cdot,\cdot\rangle$ between differential forms and multi-vectors:
\begin{equation}
\iota_\Omega \textrm{vol}(X)=\langle \textrm{vol},\Omega\wedge X\rangle,~~~~~~X\in\Gamma(T\Sigma).
\end{equation} 
In this case the Jacobi identity is interpreted as
\begin{equation}
\label{Jcb}
\mathcal{J}\wedge d\mathcal{J}=0.
\end{equation}

Now, let us introduce the bi-vector on the contact metric manifold \\
$(\Sigma,\phi, \xi_2,\eta^2, g)$ as 
\begin{equation}
\label{Poisson}
\Omega=\xi_1\wedge\xi_2. 
\end{equation} 
Here $f_y=0$ does not necessarily hold. It follows from (\ref{e7}) that
\begin{equation}
\label{Poisson12}
\eta^3 =\imath_\Omega\textrm{vol}_g,
\end{equation} 
where $\textrm{vol}_g = \eta^1\wedge\eta^2\wedge\eta^3$ is the volume form on $\Sigma$.
From (\ref{Jcb}) and the structure equations (\ref{strcture}) we obtain
\begin{equation}\label{Poneform}
\eta^3\wedge d\eta^3= 2f_y \textrm{vol}_g.
\end{equation} 
It follows that the Jacobi identity for  the bi-vector $\Omega=\xi_1\wedge\xi_2$ is described  simply by the  linear partial differential equation $f_y=0$. Thus, the bi-vector (\ref{Poisson}) is a Poisson structure or the 1-form (\ref{Poisson12}) is a Poisson 1-form on $(\Sigma,\phi,\xi_2,\eta^2, g)$ if and only if $f_y=0$. Together with the Theorem (\ref{thm2}) the following is immediate.

\begin{theorem}
	Let $(\Sigma,\phi,\xi_2,\eta^2, g)$ be the contact metric manifold  associated with a second order ODE $y''= f(x,y,y')$. The bi-vector $\Omega=\xi_1\wedge\xi_2$ on $\Sigma$
	defines a Poisson structure iff $(\phi, \xi_2,\eta^2, g)$ is a Sasakian structure on $\Sigma$.
\end{theorem}

\subsection{Bi-Hamiltonian structure}
The Hamiltonian vector field $v=\Omega (dH,\cdot)$ associated with the Poisson structure $\Omega=\xi_1\wedge\xi_2$ on the Sasakian manifold $(\Sigma,\phi, \xi_2,\eta^2, g)$ is given in terms of  $\xi_i$ as 
\begin{equation}
\label{Hvf}
v=\Omega (dH,\cdot)= \xi_1(H)\xi_2-\xi_2 (H)\xi_1.
\end{equation} 
To find the coordinate expression of $v$, if we substitute $
\xi_1= 2 ({\partial_x}+p{\partial_y}+f {\partial_p})$ and $\xi_2=2{\partial_y}$ we get
\begin{equation}
\label{Hvf2}
v=- 4\left[H_y \partial_x-(H_x+fH_p)\partial_y+fH_y\partial_p\right].
\end{equation}
The equations of motion are given in coordinates as
\begin{eqnarray}
\dot{x}&=&-4H_y\notag\\
\dot{y}&=&4(H_x+fH_p)\label{systemy}\\
\dot{p}&=&-4fH_y.\notag
\end{eqnarray}
Clearly we have $v(H)=0$. 
If we consider the definition of the cross product 
\begin{equation}
g(X\times Y, Z)=\textrm{vol}_g (X,Y,Z),
\end{equation}
then we see that \begin{equation}
\xi_i\times \xi_j=\epsilon_{ijk}\xi_k,
\end{equation}
where $\epsilon_{ijk}$ is the Levi-Civita symbol. Thus $\xi_1,\xi_2,\xi_3$ defines a right-handed orthonormal basis for $T_x \Sigma$ at a point $x$. It follows from here that the Hamiltonian vector field (\ref{Hvf})  is written as  
\begin{equation}
\label{Hvf31}
v= J\times \nabla H.
\end{equation} 
where $J=(\eta^3)^\sharp=\xi_3$ is called the Poisson vector field corresponding to the Poisson structure $\Omega=\xi_1\wedge\xi_2$ and $\nabla H = \sum_i \xi_i(H)\xi_i$ is the gradient of Hamiltonian function $H$ with respect to the metric (\ref{metric}). Here the musical isomorphism $\sharp: T^* \Sigma\rightarrow T\Sigma$ is defined by $g(\omega^\sharp, X)=\omega(X)$ for any 1-form $\omega$ and any vector field $X$. The inverse of $\sharp$ is denoted by $\flat$. For our purposes we shall give the definition of a bi-Hamiltonian structure of a vector field $v$ and the compatibility in following form \cite{Olver1990}:

\begin{definition}
	A dynamical system $\dot{x}=v$ is said to be \textit{bi-Hamiltonian} if it can be written in the Hamiltonian form in two distinct ways:
	\begin{equation}\label{biHamdef}
	v= \Omega_1 (dH_2, .)=\Omega_2 (dH_1, .),
	\end{equation} 
	such that the Poisson structures $\Omega_1$ and $\Omega_2$ are nowhere multiples of each other. This bi-Hamiltonian structure is said to be \textit{compatible} if $\Omega_1+\Omega_2$ is also a Poisson structure.  
\end{definition}

From the Hamilton's equations (\ref{Hvf31}) we see that the Poisson vector field is perpendicular to the Hamiltonian vector field. Let us introduce the bi-vector 
\begin{equation}\label{omega2bivect}
\Omega_2 = \mu\xi_3\wedge\xi_1+\lambda\xi_2\wedge\xi_3,
\end{equation} 
where $\lambda = \xi_1(H)$ and $\mu = \xi_2(H)$, and prove the following theorem:
\begin{theorem}
	Let $(\Sigma,\phi,\xi_2,\eta^2, g)$ be the contact metric manifold associated with a second order ODE $y''= f(x,y,y')$. Then  the bi-vector (\ref{omega2bivect})
	defines a Poisson structure on $\Sigma$ iff $\mu=0$.
\end{theorem}

\begin{proof}
	The corresponding 1-form is obtained by $\eta=\imath_{\Omega_2}\textrm{vol}_g$ as
	\begin{equation}
	\eta= \lambda\eta^1+\mu\eta^2 =dH-\xi_3(H)\eta^3,
	\end{equation}  
	where $\xi_3(H)=2H_p$. From
	\begin{equation}\label{Poneform2}
	\eta\wedge d\eta= \left(\mu\xi_3(\lambda)-\lambda\xi_3(\mu)-2\mu^2\right) \textrm{vol}_g
	\end{equation} 
	it follows that $\eta$ is a Poisson 1-form iff
	\begin{equation}\label{Jacobiid}
	\mu\xi_3(\lambda)-\lambda\xi_3(\mu)-2\mu^2=0.
	\end{equation}
	If $\mu=0$, then $\eta$ is a Poisson 1-form. Conversely,  assume that $\eta$ is a Poisson 1-form and $\mu\neq 0$. Then (\ref{Jacobiid}) takes the form $\partial_p (\lambda/\mu)=1$ and hence we write 
	\begin{equation}
	\frac{\lambda}{\mu}= p-\psi(x,y)
	\end{equation}
	for some smooth function $\psi=\psi(x,y)$.
	Substituting $\lambda=\xi_1(H)=2(H_x+pH_y+fH_p)$ and $\mu=\xi_2(H)=2H_y$ into here we get
	\begin{equation}\label{conservative0}
	H_x+fH_p= -H_y\psi(x,y)
	\end{equation}
	Since $H_y\neq 0$,  
	eliminating $dt$ from (\ref{systemy}) we may write
	\begin{equation}\label{eliminatet}
	\frac{dy}{dx}=-\frac{H_x+fH_p}{H_y}=p,~~~~\frac{dp}{dx}=\frac{fH_y}{H_y}=f.
	\end{equation}
	It follows from (\ref{eliminatet}) and (\ref{conservative0})
	that $dy/dx=\psi(x,y)$ and  $f= \psi_x+\psi\psi_y$. 
	Using these and writing (\ref{conservative0}) as
	\begin{equation}\label{conservative}
	(\partial_x+\psi(x,y)\partial_y+f(x,p)\partial_p)H=0
	\end{equation}
	we obtain
	\begin{equation}
	\frac{d}{dx}H\left(x,y(x),p\left(x,y(x)\right)\right)=0.
	\end{equation}
	This implies that $\lambda =0$ and hence the Jacobi identity is given by $\eta\wedge d\eta=-2\mu^2\textrm{vol}_g=0$ which contradicts  $\mu\neq 0$. \qed	
\end{proof}	 
We now have two Poisson structures $\Omega_1=\Omega=\xi_1\wedge\xi_2$ and $\Omega_2 = \lambda\xi_2\wedge\xi_3$ and the Hamiltonian vector field
\begin{equation}
\label{Hvf3}
v= \xi_1(H)\xi_2
\end{equation}
on the Sasakian manifold $(\Sigma,\phi,\xi_2,\eta^2, g)$. Note that the Hamiltonian vector field $v$ is a multiple of the Reeb vector field.
We shall show that $v$ can be written in bi-Hamiltonian form with respect to the Poisson structures $\Omega_1=\xi_1\wedge \xi_2$ and $\Omega_2 = \lambda\xi_2\wedge\xi_3$ such that this bi-Hamiltonian structure is compatible.
We may write $v= \xi_1(H)\xi_2$ as
\begin{equation}
v= \xi_3\times \nabla H
\end{equation} 
where $H = H(x,p)$ is the Hamiltonian function with $\xi_1(H)\neq 0$ and $\xi_3=(\eta^3)^\sharp$ is the Poisson vector field corresponding to the Poisson structure $\Omega_1=\xi_1\wedge\xi_2$. This suggests that on the Sasakian manifold  $(\Sigma,\phi,\xi_2,\eta^2, g)$ the right-hand side of $d^2y/dx^2 = f(x,p)$ can be taken as Hamiltonian function, i.e. $H =f$ whenever $f_x+ff_p\neq 0$, where $p=y'$. In this case  (\ref{Hvf3}) is given by $v= 2(f_x+ff_p)\xi_2$ 
and from (\ref{geo1}) we have $\nabla_v v=0$, that is, integral curves of the Hamiltonian vector field $v$ are geodesic curves on $\Sigma$. As is well known an integral curve of the Reeb vector field is a geodesic curve but this is not true for an arbitrary multiple of the Reeb vector field.

If we introduce the function $H_1=-\frac{1}{2}p$, then 
we get
\begin{equation}
\nabla H_1 = -f\xi_1-\xi_3,
\end{equation}
and
\begin{equation}
v=  J_2\times \nabla H_1,
\end{equation} 
where
\begin{equation}
J_2= \xi_1(H)\xi_1
\end{equation}
is the Poisson vector field corresponding to the Poisson structure $\Omega_2 = \lambda\xi_2\wedge\xi_3$ and is defined by $J_2=\eta^\sharp$. Here $\eta=\xi_1(H)\eta^1$. This shows that we have the bi-Hamiltonian structure for $v$, that is,
\begin{equation}
v= \xi_3\times \nabla H= J_2\times \nabla H_1.
\end{equation} 
Since $H_y=0$, this bi-Hamiltonian structure is compatible, that is, in terms of the corresponding Poisson 1-forms the compatibility condition holds identically:
\begin{equation}
(\eta^3+\eta)\wedge d(\eta^3+\eta)=\eta\wedge d\eta^3+\eta^3\wedge d\eta=\eta^3\wedge d\eta=0.
\end{equation} 
As a result we have the following:
\begin{theorem}
	Let $(\Sigma,\phi,\xi_2,\eta^2, g)$ be the Sasakian manifold associated with a second order ODE $y''= f(x,y')$,where $f_x+ff_p\neq 0$. Then the vector field $v= \xi_1(H)\xi_2$ is written in bi-Hamiltonian form 
	\begin{equation}
	v=\Omega_1 (dH,\cdot)=\Omega_2 (dH_1,\cdot), ~~~~H=H(x,p),~H_1=-\frac{1}{2}p
	\end{equation}
	with respect to the Poisson structures $\Omega_1=\xi_1\wedge\xi_2$, $\Omega_2 = \xi_1(H)\xi_2\wedge\xi_3$ such that this bi-Hamiltonian structure is compatible.
\end{theorem}

\subsection{Bi-Hamiltonian structure of the Reeb vector field}
As a particular case we shall  investigate a bi-Hamiltonian structure of the Reeb vector $\xi_2$ on the Sasakian manifold $(\Sigma,\phi,\xi_2,\eta^2, g)$ associated with a second order ODE $y''= f(x,y')$.

From the structure equations (\ref{strcture}) we see that $\eta^3$ canonically defines a Poisson 1-form when $f_p=0$. In this case we have
\begin{eqnarray}\label{strctureH}
d\eta^1&=&0 \nonumber \\
d\eta^2 &=& 2\eta^1\wedge\eta^3 \\
d\eta^3 &=& 0. \nonumber
\end{eqnarray}
Since all the structure functions are constants, they can be identifed with the structure constants of a (local) Lie group $G$ of transformations of $\Sigma$ preserving the coframe $\eta_g = (\eta^1,\eta^2,\eta^3)^t$. In this case the 1-forms $\eta^1,\eta^2,\eta^3$ will
correspond to a basis for the left-invariant Maurer-Cartan forms on $G$ and $\eta_g$ defines a rank zero coframe
or an invariant coframe on $G$ \cite{Olver1993,Olver1995}. As a consequence, the structure equations (\ref{strctureH})  manifest that $\Sigma$ is locally diffeomorphic to the Heisenberg group and $\eta_g$ is identified with the Maurer-Cartan coframe. In this case the structure equations are regarded as the Maurer-Cartan equations of the Heisenberg group. Since $\eta^i$'s are invariant 1-forms, Riemannian metric (\ref{metric}) defines an invariant metric  on $\Sigma$ \cite{Milnor1976}.

Since $f_y=f_p=0$, the compatible bi-Hamiltonian structure for the left-invariant Reeb vector field is given by 
\begin{equation}\label{hsb}
\xi_2= \xi_3\times \nabla H = J_2\times \nabla H_1,
\end{equation} 
where $H= \frac{1}{2}x$, $J_2 = \nabla H=\xi_1$ and $H_1 =\frac{1}{2}\int^x fdx' -\frac{1}{2}p$. This shows that the independent variable $x$ defines a Hamiltonian function for $\xi_2$. The contact 1-form $\eta^2=1/2(dy-pdx)$ defines a nonintegrable distribution $\mathscr{D}= \ker\eta^2$  in $\Sigma$ and an integral curve of (\ref{hsb}) corresponds to a non-horizontal geodesic on $\Sigma$. It is also suitable to note that since (1,1) tensor $\phi$ is a linear transformation on a horizontal subspace, the bi-Hamiltonian structure of the Reeb vector field is also given by means of $\phi$.

\section*{Concluding remark}

As a final remark let us consider the normal bundle of an integral curve of the Hamiltonian vector field $v= \xi_1(H)\xi_2$  as a two dimensional real vector bundle $\mathscr{D}= \ker\eta^2$ over $\Sigma$. Together with the fiber-wise defined Riemannian metric $g_\mathscr{D} = \varpi^1\otimes\varpi^1+\varpi^2\otimes\varpi^2$, each fiber is spanned by the orthonormal vector fields $\xi_1$ and $\xi_3$, where $\varpi^1 =\eta^1$ and  $\varpi^2 =\eta^3$. The structure equations result in
\begin{equation}
d\varpi^i = - \omega^i_j\wedge\varpi^j,
\end{equation} 
where 
\begin{equation}
(\omega^i_j)= \left(\begin{array}{cc}0 & -2f_p \varpi^2 \\
2f_p \varpi^2 & 0
\end{array}\right)
\end{equation} 
defines connection form in the vector bundle $\mathscr{D}\rightarrow \Sigma$. The (1,1) tensor $\phi$ defines complex structure on the fibers of $\mathscr{D}$ and hence  $\mathscr{D}$ can be regarded as a one-dimensional complex line bundle. The complex connection form is defined to be $1\times 1$ matrix $(i\omega^1_2)$ with corresponding curvature matrix $(i\Omega^1_2)$ on this line bundle. In this case $(-1/2\pi i)\text{tr} (i\Omega^1_2)$ represents the first Chern class $c_1(\mathscr{D})$ which is the characteristic cohomology class in $H^2(\Sigma,\mathbb{R})$ \cite{Morita2001,Milnor1974}. Explicit computation shows that  $\Omega^1_2 = -4(f_x+ff_p)_p \varpi^1\wedge \varpi^2$ and hence the first Chern class is trivial iff the curvature vanishes. Note that this curvature determines the sectional curvature associated with the plane in $T_q\Sigma$ spanned by $\xi_1$ and $\xi_3$  at a point $q$ \cite{Bayrakdar2018}.  Recently, it has been shown in \cite{IsimEfe2017} that an autonomous dynamical system defined by a nonvanishing vector field on an orientable three-dimensional manifold is globally bi-Hamiltonian
if and only if the first Chern class of the normal bundle of the given vector field vanishes. It follows that the dynamical system determined by the Hamiltonian vector field $v= \xi_1(H)\xi_2$ on the Sasakian manifold corresponding to an ODE of the form $d^2y/dx^2 = f(x,p)$ is globally bi-Hamiltonian if and only if $f$ satisfies $f_x+ff_p = \Psi (x)$ for some function $\Psi$. 

\section*{Acknowledgment} The author express his thanks to the anonymous referee for his/her comments
and suggestions that helped to improve the paper.

\end{document}